\documentclass{amsart}
\usepackage{amssymb}
\usepackage{enumerate, xspace}
\usepackage{dsfont}
\usepackage{amsfonts}
\usepackage{xcolor}
\usepackage{graphicx}
\graphicspath{ {./} }
\usepackage{amsmath}
\usepackage[normalem]{ulem}
\usepackage[colorlinks]{hyperref}
\theoremstyle{plain}

\newtheorem{lemma}{Lemma}

\newtheorem{proposition}{Proposition}
\newtheorem{remark}{Remark}

\newtheorem{theorem}{Theorem}
\numberwithin{equation}{section}

\let\eps=\varepsilon

\begin{document}
\title{Linear response due to singularities}

\author{Wael Bahsoun}

\address{Department of Mathematical Sciences, Loughborough University,
Loughborough, Leicestershire, LE11 3TU, UK}
\email{W.Bahsoun@lboro.ac.uk}

\author{Stefano Galatolo} 
\address{Dipartimento di Matematica, Universita di Pisa, Via \ Buonarroti 1,Pisa -
Italy}
\email{galatolo@dm.unipi.it}

\date{\today}
\thanks{}
\begin{abstract} 
 
\noindent {It is well known that  a family  of tent-like maps with bounded derivatives has no linear response for typical deterministic perturbations changing the value of the turning point. In this note we prove the following result: if we consider a tent-like family with a  \emph{cusp} at the turning point, we recover the linear response.
 More precisely,  let  $T_\eps$ be a family of such cusp maps generated by changing the value of the turning point of $T_0$ by a deterministic perturbation and let $h_\eps$ be the corresponding invariant density. We  prove that $\eps\mapsto h_\eps$ is differentiable in $L^1$ and provide a formula for its derivative.}
\end{abstract}
\maketitle

\section{Introduction}

Let $M$ be a compact manifold with a reference volume and $T:M\rightarrow M$ be a map, whose
iterates determine the dynamics. A $T$-invariant measure $\mu$; i.e. $%
T_*\mu=\mu$, is said to be \emph{physical} if there is a positive volume set 
$B$ such that for any continuous observable $f:M\rightarrow \mathbb{R}$%
\begin{equation*}
\underset{n\rightarrow \infty }{\lim }\frac1n\sum_{i=0}^{n-1}f(T^{i}(x)) =\int_{M}f~d\mu
\end{equation*}
for all $x\in B$. An important question is to study how such measures vary, in
an appropriate topology, under suitable perturbations of $T$. Namely,
consider a family $\{(M,T_{{\varepsilon} })\}_{{\varepsilon}\in V}$, where $V
$ is a small neighbourhood of $0$, $T_0:=T$ and $T_{{\varepsilon}
}\rightarrow T_{0}$, as ${\varepsilon} \rightarrow 0$, in some suitable
topology. For many chaotic systems it turns out that $T_{%
\varepsilon}$ admits a unique physical measure $\mu _{{\varepsilon} }$.

The system $(M,T_{0},\mu _{0})$ is called statistically stable if the
map ${\varepsilon} \longmapsto \mu _{{\varepsilon} }$ is continuous, in an
appropriate topology, at ${\varepsilon} =0$. Quantitative
statistical stability is provided by quantitative estimates on its modulus
of continuity (see e.g. \cite{A,K, Gpre, Gcsf, V, BLK}). In the case where ${\varepsilon} \longmapsto \mu _{{%
\varepsilon} }$ is differentiable in some sense, the system is also said to
admit \emph{linear response}; i.e., the `derivative' $\dot{\mu}_0$
represents the first order term of the response of the system to the
perturbation 
\begin{equation*}
\mu _{\varepsilon }=\mu _{0}+\dot{\mu}_0{\varepsilon} +o({\varepsilon} )
\end{equation*}%
where the error { term} is understood in an appropriate topology.

Linear response results in the context of deterministic dynamics was first
obtained in the case of uniformly hyperbolic systems \cite{R} (see also \cite{KP}). Nowadays linear response results are known
for systems outside the uniformly hyperbolic setting (see \cite{BS, BT,D,GS, K},
the survey article \cite{B} and references therein.) For systems with
discontinuities or critical points the situation is quite
complicated even in the presence of uniform expansion. Indeed, for suitable small
perturbations of piecewise expanding maps there are several examples that lack
statistical stability or linear response (see \cite{K, M, b2} and \cite{AK} for recent results in this direction). In the literature there are some results indicating that perturbations
which are not changing the \emph{topological class}\footnote{%
A perturbation does not change the topological class if it changes a system
to a system which is topologically conjugated to it} or tangent to the
topological class, linear response is likely to occur, while for
perturbations which are transversal to the topological class there is no
linear response (see \cite{BS} for the case of piecewise expanding unimodal
maps, \cite{BBS} and \cite{BS2} for smooth unimodal maps and \cite{GS} for
results along this line in the case of rotations).

In this paper we study perturbations of one dimensional tent-like piecewise expanding maps with unbounded derivatives at the turning point (a cusp singularity). We prove linear response for a large class of deterministic perturbations also changing the image at the turning point and hence the topological class. The existence of such a linear response is due to the singularity. Indeed, unlike the usual tent maps with bounded derivative, which do not admit linear response for the same kind of perturbations (\cite{B,M}), linear response in the singular case is attained due to fact that the cusp has a `regularizing effect' at the level of the action of the associated transfer operators. 
The present paper  studies this regularization phenomenon due to a singularity, in the relatively simple case of tent-like maps. We believe however that this phenomenon might also appear in other important classes of dynamical systems and this motivates further studies in this direction.

 Cusp like  singularities and unbounded derivatives appear in several important systems including Lorenz-type maps \cite{BMR, BR} and  billiard maps (see e.g. \cite{BDL18}). Unlike the case of Anosov flows \cite{BL} where linear and higher order response is proved, in the case of the classical Lorenz flow only statistical stability is known \cite{AS,BMR, BR} although numerical evidence suggests that the Linear Response may hold in this case (\cite{CR}, \cite{SCW}).
  By highlighting the role of singularities in studying linear response, we hope that the results of the present paper contribute to the understanding of linear response for Lorenz-like flows and billiard maps.

The paper is organised as follows. In Section \ref{Sec2} we introduce the
class of systems we consider and state the main result, Theorem \ref{thm:main}, of the paper. In Section  \ref{Sec3} we prove Theorem \ref{thm:main} in a series of lemmas. In Section \ref{example} a concrete family of maps satisfying the assumptions of Section \ref{Sec2}. Concluding remarks are presented in Section \ref{Sec5}.

\noindent {\bf Acknowledgements.} The research of W. Bahsoun is supported by EPSRC grant EP/V053493/1. W. Bahsoun would like to thank the University of Pisa for its hospitality during his visit to S. Galatolo. The research of S.G. was partially supported by  the
research project "Stochastic properties of dynamical systems" (PRIN 2022NTKXCX) of the Italian Ministry of Education and Research and by a grant from MIUR (Dipartimenti di Eccellenza DM 11/05/2017, n. 262).
The authors thank M. Ruziboev,  D. Smania and V. Baladi for fruitful discussions and for suggesting useful references. The authors also thank anonymous referees for careful reading and useful comments.

\section{Family of maps and statement of the main result}\label{Sec2}

For $\eps \in [0,\delta)$, $\delta >0$, consider a family $T_\eps:[0,1]\rightarrow \lbrack 0,1]   $ of nonsingular maps, with respect to Lebesgue measure $m$ on $[0,1]$. The transfer operator associated with $T_\eps$, denoted by $L_{ T_\eps}$, is defined by duality as follows: for $f\in L^1, g\in L^\infty$
$$\int_0^1L_{T_\eps} f\cdot g dm=\int_0^1f\cdot g\circ T_\eps dm.$$
Let $C\ge0$, $c\in (0,1)$, $\beta\in (-1,-\frac34)$,
we assume that $T_\eps$ satisfies the following assumptions:

\begin{enumerate}
\item[(A1)] $T_{0,\eps}:=T_\eps |_{[0,c)}$  and $T_{1,\eps}:=T_\eps |_{(c,1]}$ are one to one.

\item [(A2)] $T_\eps(0)=0,T_\eps(1)=0,\lim_{x\to c^{\pm}}T_\eps(x)=a_\eps\in \lbrack 0,1]$

\item[(A3)] $T_\eps|_{[0,1]\setminus\{c\}}\in C^{3}$.

\item[(A4)] \label{teta} $\sup_{\varepsilon\in[0,\delta)}\inf_{x\in \lbrack 0,1]\setminus\{c\}}|T_\eps^{\prime }(x)|\geq \theta>1$.

\item[(A5)] $T_0$ is topologically mixing on $[0,a_0]$.\footnote{See \cite{V}, (E3) page 46 for details.}

\item[(A6)] { $\forall \eps\in [0,\delta) $ } $\lim_{x\rightarrow c^{\pm }}T_\eps^{\prime }(x)=\pm \infty $ and 
$\lim_{x\rightarrow c^{\pm }}\frac{|T_\eps^{\prime }(x)|}{|x-c|^{\beta }}=C_{\varepsilon,1}$.

\item[(A7)]{  $\forall \eps\in [0,\delta) $} $\lim_{x\rightarrow c^{\pm }}\frac{|T_\eps^{\prime \prime }(x)|}{
|x-c|^{\beta -1}}=C_{\varepsilon,2}$, $\lim_{x\rightarrow c^{\pm }}\frac{%
|T_\eps^{\prime \prime \prime }(x)|}{|x-c|^{\beta -2}}=C_{\varepsilon,3}$. 

\item [(A8)]\label{unif}{  The power law divergence of the derivatives of $T_\eps$ is uniform in $\varepsilon$: $$\sup_{\eps\in[0,\delta), i\in\{0,1,2\}, x\in[0,1]}|\frac{T_\eps^{(i+1)}(x)}{
(x-c)^{\beta -i}}|<\infty.$$ Furthermore, $\inf_{\eps\in[0,\delta),x\in[0,1]}|\frac{T_\eps^{\prime}(x)}{
(x-c)^{\beta}}|>0.$}

\item[(A9)] For\footnote{We use the notation $W^{i,1}$, $i=1,2$ to denote the usual Sobolev spaces equipped with the norms $\|f\|_{W^{i,1}}=\sum_{k=0}^{i}\|f^{(k)}\|_1$. Also in this assumption we consider the weak norm to be $L^2$. The reason for this will become apparent in Section \ref{Sec3}.} $f\in W^{1,1}$ we require that the transfer operators are close in a mixed norm when $\eps$ is small: 
\begin{equation}\label{near} 
\sup_{\|f\|_{W^{1,1}}\le 1}\|(L_{T_0}-L_{T_\eps})f\|_{ L^2}\to0 
\end{equation}
as $\eps\to 0$.
Moreover, for  $f\in W^{2,1}$, there exists $\partial
_{\eps }L_{T_\eps}f|_{\eps =0 } \in W^{1,1}$ 
such that
\begin{equation}\label{eq:der}
\Vert L_{T_\eps}f-L_{T_0}f-{\varepsilon }(\partial
_{\eps }L_{T_\eps}f|_{\eps =0 })\Vert
_{W^{1,1}}=o(\eps).
\end{equation}
\end{enumerate}
A concrete family of maps satisfying assumptions $(A1),...,(A9)$ is presented in Section \ref{example}, see also Figure \ref{fig:mesh1} for an example of a typical graph in such a family.

Although it is known that $L_{T_\eps}$ admits a spectral gap when acting on the space of functions of generalized  bounded variations \cite{K1},  we study the action of $L_{T_\eps}$ on finer Banach spaces. Namely, the Sobolev spaces $W^{i,1}$, $i=1,2$. In particular, we show that $L_{T_\eps}$ admits a spectral gap on $W^{i,1}$, $i=1,2$. This will allow us to conclude that $T_\eps$ admits an invariant density which is regular enough, in $x$. The latter is essential to derive the linear response formula that we are after in this work. The following result is the main result of the paper.
\begin{theorem}
\label{thm:main} {There is $\delta_2 >0$ such that for $\eps\in[0,\delta_2 )$} $T_\eps$ admits a unique invariant probability density $h_\eps\in W^{2,1}$; moreover, ${\varepsilon}\mapsto h_{\varepsilon}$ is differentiable in $L^1$ at $\eps=0$. In particular, 
\begin{equation*}
h_{\varepsilon} = h_0 + {\varepsilon}(I-{L_{T_0}})^{-1}(q) + o({\varepsilon}%
),
\end{equation*}
here 
\begin{equation}\label{eq:q}
q(x)=\left\{ 
\begin{array}{lr}
L_{T_0} [A_0 { h_0}^{\prime }+ B_0 {h_0}](x) &\text { for }x\in [0,a_0)\\
0&\text{ for } x\in  [a_0,1]
\end{array}
\right. ,  
\end{equation}
with
\begin{equation*}
A_{\varepsilon}=-\left(\frac{\partial_{{\varepsilon}}T_{{\varepsilon}}}{
T^{\prime }_{{\varepsilon}}}\right),\hskip 0.5cm B_{\varepsilon}= \left(%
\frac{\partial_{{\varepsilon}} T_{{\varepsilon}}\cdot T_{{\varepsilon}%
}^{\prime \prime }}{T_{\varepsilon}^{\prime 2}}-\frac{\partial_{{\varepsilon}%
} T_{{\varepsilon}}^{\prime }}{ T_{\varepsilon}^{\prime }}\right)
\end{equation*}
and the $o$ is in the $L^1$-topology.
\end{theorem}

\section{Proof of Theorem \ref{thm:main}}\label{Sec3}
In this section we prove Theorem \ref{thm:main} in two steps. First, in subsection \ref{subsec:spec} we show that $L_{T_\eps}$ admits a uniform, in $\eps$, spectral gap when acting on $W^{1,1}$ and $W^{2,1}$. This implies, in particular, that $T_\eps$ has a unique invariant density $h_\eps\in W^{2,1}$. Then in subsection \ref{subsec:resp} we show that $\eps\mapsto h_\eps$ is differentiable in $L^1$ and obtain a formula for the derivative.

\subsection{Uniform spectral gap on  $W^{1,1}$ and on $W^{2,1}$}\label{subsec:spec}
In this section we prove that the transfer operators associated with our class of systems admit a uniform spectral gap when acting on suitable Sobolev spaces (Lemma \ref{lem:spec}). 
We obtain this as a consequence  of classical results which are recalled below.

\subsubsection{Classical results}
We first recall Lemma 2.2 of \cite{bgk} \ which is a modification of Theorem \
XIV.3 of \cite{HH}.

\begin{lemma}\label{BGK}
Let $(B,\|~\|_{s})$ be a Banach space, let $\|~\|_{w}$ be a continuous
semi-norm on $B$ and $Q$ a bounded linear operator on $B$, such that for any sequence $x_{n}$
with  $\|x_{n}\|_{s}\leq 1$ it contains a Cauchy subsequence for $\|~\|_{w}$. 

Assume there is $\lambda \geq 0$ and $C>0$ such that, for any $f\in B$
\begin{equation*}
\|Qf\|_{s}\leq \lambda \|f\|_{s}+C\|f\|_{w}.
\end{equation*}%
Then the essential spectral radius of $Q$ is bounded by $\lambda $.
\end{lemma}

We also recall the main result of \cite{KL}, stating it in a simplified form suitable for our purposes.

\begin{lemma}\label{KL1}
Let $(B,\|~\|_{s})$ and $\|~\|_{w}$ as above, and $P_\varepsilon: B\to B$ with $\varepsilon \geq 0$ be a family of bounded linear operators.  Assume: there are $C_1>0 $ and $M\geq1$ such that
\begin{equation}\|P_\varepsilon^n\|_w\leq C_1M^n;
\end{equation}
there are  $C_2,C_3>0$ and  $0\leq \lambda < 1$  such that, for any $f\in B$, $\varepsilon \geq0$, $n\in \mathbb{R}$
\begin{equation}\label{ULY}
\|P_\varepsilon^n f\|_{s}\leq C_2 \lambda^{n} \|f\|_{s}+C_3M^n\|f\|_{w};
\end{equation}
$1$ is not in the spectrum of $P_\varepsilon$ for any $\varepsilon\ge 0$, 
\begin{equation}\label{eq:KLmixed}
\|P_0-P_\eps \|_{s\to w}\leq \tau(\varepsilon)
\end{equation}
where $\tau(\varepsilon) \to 0 $ monotonically and upper semicontinuously. Then there are $\varepsilon_0, a>0$ such that for all $0\leq \varepsilon  \leq \varepsilon_0$ and $f\in B$ 
\begin{equation}
\|(Id-P_\varepsilon )^{-1}f\|_s\leq a\| f\|_s
\end{equation}
and
\begin{equation}
\lim_{\varepsilon \to 0}\|(Id-P_0 )^{-1}-(Id-P_\varepsilon )^{-1}\|_{s\to w}=0.
\end{equation}
\end{lemma}

\subsubsection{Uniform spectral gap for the associated transfer operators}

Let $$W^{1,1}_0=\{f\in W^{1,1}: \int_0^1f dm=0\}.$$
We now state the main result of this section.
\begin{proposition}\label{lem:spec} { There is $\delta_2>0$ such that for any $\varepsilon \in [0,\delta_2) $}
$L_{T_\eps}$ admits a unique invariant density $h_\eps \in W^{2,1}$. Furthermore, $L_{T_\eps}$  has a uniform, in $\varepsilon$, spectral gap when acting on $W^{1,1}$ and $W^{2,1}$. In particular, $\exists$  $C>0$ such that for all $\varepsilon \in [0,\delta)$ 
\begin{equation}\label{resnorm}\|(Id-L_{T_\eps} )^{-1}\|_{W^{1,1}_0\to W^{1,1} }\leq C.
\end{equation}
\end{proposition}
To prove Proposition \ref{lem:spec} we apply Lemma \ref{BGK} and Lemma \ref{KL1}, we first prove few lemmas that will be used for this purpose and we start with an auxiliary lemma that verifies condition \ref{eq:KLmixed} in Lemma \ref{KL1}. First, consider the pointwise representation of the transfer operator associated with $T_\eps$  
\begin{equation}
(L_{T_\eps}f)(x)=\left\{ 
\begin{array}{c}
\sum_{y\in T_\eps^{-1}(x)}\frac{1}{|T_\eps^{\prime }(y)|}f(y)~if~x\in [0,a_\varepsilon] ;\\ 
0~if~x\in ( a_\varepsilon,1].%
\end{array}%
\right.  \label{rt}
\end{equation}

\begin{lemma}\label{lem:LYW11}
{ There are $0<\lambda<1,M\geq0$ such that
for any $\eps \in[0,\delta),$ $f\in W^{1,1}$ we have
$$\|(L_{T_\eps}f)^{^{\prime }}\|_{1}
\leq \lambda \|f^{\prime }\|_{1}+M\|f\|_{2}.$$}
\end{lemma}
\begin{proof}
By \eqref{rt} for $f\in W^{1,1}$, we have
\begin{equation}\label{der1}
(L_{T_\eps}f)^{^{\prime }}(x)=\left\{ 
\begin{array}{c}
\sum_{y\in T_\eps^{-1}(x)}\frac{f^{\prime }(y)}{|T_\eps^{\prime }(y)|T_\eps^{\prime }(y)}-\frac{%
T_\eps^{\prime \prime }(y)}{(T_\eps^{\prime }(y))^{3}}f(y)~if~x\in [0,a_\varepsilon] ;\\ 
0~if~x\in ( a_\varepsilon,1].
\end{array}%
\right. 
\end{equation} 
Notice that $L_{T_\eps}f \in W^{1,1}$. Indeed since $f$ is bounded,\ by $\eqref{rt}$ $\lim_{x\rightarrow a_{\varepsilon}}L_{T_\eps}f=0$. \ $L_{T_\eps}f$ is then continuous, with derivative almost everywhere. Furthermore, we note that in \eqref{der1} when $x\to a_{\varepsilon}$ then $y\to c$ and\footnote{Here the asymptotic equivalence $f\sim g$ stands for $\lim_{x\to c}\frac{|f(x)|}{|g(x)|}=C$  with $0<C<+\infty.$ } $\frac{T_\eps^{\prime \prime }(y)}{(T_\eps^{^{\prime }})^{3}(y)}\sim \frac{|y-c|^{\beta -1}}{|y-c|^{3\beta }}=|y-c|^{-2\beta-1}$ . Therefore,   
$\frac{T_\eps^{\prime \prime }(y)}{(T_\eps^{^{\prime }})^{3}(y)}\rightarrow 0$ as $
y\rightarrow c,$ this shows that $L_{T_\eps}f\in W^{1,1}.$

Now, observe that
\begin{equation}
(L_{T_\eps}f)^{^{\prime }}=L_{T_\eps}(\frac{1}{T_\eps^{\prime }}f^{\prime })-L_{T_\eps}(\frac{%
T_\eps^{\prime \prime }}{{|T_\eps^{^{\prime }}|T_\eps^{^{\prime }}}}f).  \label{preLY}
\end{equation}

By \eqref{preLY} we get
\begin{eqnarray}\label{LY1}
\|(L_{T_\eps}f)^{^{\prime }}\|_{1} &\leq &\|L_{T_\eps}(\frac{1}{T_\eps^{\prime }}f^{\prime
})\|_{1}+\|L_{T_\eps}(\frac{T_\eps^{\prime \prime }}{(T_\eps^{^{\prime }})^{2}}f)\|_{1} \\
&\leq &\|\frac{1}{T_\eps^{\prime }}f^{\prime }\|_{1}+\|\frac{T_\eps^{\prime \prime }}{
(T_\eps^{^{\prime }})^{2}}f\|_{1} \\
&\leq &\|\frac{1}{T_\eps^{\prime }}\|_{\infty }\|f^{\prime }\|_{1}+\|\frac{%
T_\eps^{\prime \prime }}{(T_\eps^{^{\prime }})^{2}}f\|_{1} \\
&\leq &\lambda \|f^{\prime }\|_{1}+\|\frac{T_\eps^{\prime \prime }}{(T_\eps^{^{\prime
}})^{2}}\|_{2}\|f\|_{2}  \label{LY1b}
\end{eqnarray}
{ where for each $\eps$,} $\lambda :=\|\frac{1}{T_\eps^{\prime }}\|_{\infty }<1$ and   $\|\frac{T_\eps^{\prime \prime }}{(T_\eps^{^{\prime }})^{2}}
\|_{2}<+\infty $. Note that furthermore by assumptions (A4) and (A8) the quantities $\|\frac{1}{T_\eps^{\prime }}\|_{\infty }$  and $\|\frac{T_\eps^{\prime \prime }}{(T_\eps^{^{\prime }})^{2}}
\|_{2}$  can be  bounded uniformly for $\eps\in[0,\delta)$.  
\end{proof}

\begin{remark}\label{rmk01}
The choice of $L^2 $ as a weak space is due to the way the H\"older inequality is applied in \ref{LY1b}, and to the fact that with the standing assumptions $(A6),...,(A8)$ imply
$\|\frac{T_\eps^{\prime \prime }}{(T_\eps^{^{\prime }})^{2}}\|_{2}$ is bounded.
By changing the weak space to another $L^p$ with $p\geq2$ one can allow different power law behavior for the singularity than the one in our assumptions and still get a Lasota Yorke inequality. Moreover, in such a setting, the compact embedding of the strong space into the weak one is still granted by the Rellich-Kondracov theorem.  
\end{remark}

\begin{lemma}\label{l:LYW21}
There exists $M\geq0 $ such that for each $\varepsilon \in[0,\delta)$
\begin{equation}\label{eq:LYW21}
\|L_{T_\eps}f\|_{W^{2,1}}\leq\lambda^{2}\|f\|_{W^{2,1}}+ M\|f\|_{W^{1,1}}
\end{equation}
where $\lambda$ is the same as in Lemma \ref{lem:LYW11}.
\end{lemma}

\begin{proof}
Differentiating $\eqref{preLY}$, we get  
\begin{equation}\label{eq:s1}
\begin{split}
(L_{T_\eps}f)^{\prime \prime }&=(L_{T_\eps}(\frac{1}{T_\eps^{\prime }}f^{\prime })-L_{T_\eps}(
\frac{T_\eps^{\prime \prime }}{{|T_\eps^{^{\prime }}|T_\eps^{^{\prime }}}}f))^{\prime } \\
&=(L_{T_\eps}(\frac{1}{T_\eps^{\prime }}f^{\prime }))^{\prime }-(L_{T_\eps}(\frac{T_\eps^{\prime \prime }}{|T_\eps^{^{\prime }}|T_\eps^{^{\prime }}}f))^{\prime } \\
&=L_{T_\eps}(\frac{1}{T_\eps^{\prime }}(\frac{1}{T_\eps^{\prime }}f^{\prime })^{\prime
})-L_{T_\eps}(\frac{T_\eps^{\prime \prime }}{|T_\eps^{^{\prime }}|T_\eps^{^{\prime }}}\frac{1}{T_\eps^{\prime }}
f^{\prime }) \\
&-[L_{T_\eps}(\frac{1}{T_\eps^{\prime }}(\frac{T_\eps^{\prime \prime }}{|T_\eps^{^{\prime }}|T_\eps^{^{\prime }}}f)^{\prime })-L_{T_\eps}(\frac{T_\eps^{\prime \prime }}{|T_\eps^{^{\prime }}|T_\eps^{^{\prime }}} \frac{T_\eps^{\prime \prime }}{|T_\eps^{^{\prime }}|T_\eps^{^{\prime }}}f)]
\end{split}
\end{equation}
where
\begin{equation}\label{eq:s2}
\begin{split}
L_{T_\eps}(\frac{1}{T_\eps^{\prime }}(\frac{1}{T_\eps^{\prime }}f^{\prime })^{\prime
})-L_{T_\eps}(\frac{T_\eps^{\prime \prime }}{|T_\eps^{\prime }| T_\eps^{\prime }}\frac{1}{T_\eps^{\prime }} f^{\prime }) &=L_{T_\eps}(\frac{1}{T_\eps^{\prime }}(\frac{-T_\eps^{\prime \prime }}{(T_\eps^{^{\prime }})^2}f^{\prime }+\frac{1}{T_\eps^{\prime }}f^{\prime \prime
}))-L_{T_\eps}(\frac{T_\eps^{\prime \prime }}{|T_\eps^{^{\prime }}|T_\eps^{^{\prime }}}\frac{1}{T_\eps^{\prime }}f^{\prime }) \\
&=L_{T_\eps}((\frac{1}{T_\eps^{\prime }})^{2}f^{\prime \prime }))-L_{T_\eps}([\frac{
T_\eps^{\prime \prime }}{(T_\eps^{^{\prime }})^{2}}+{ \frac{T_\eps^{\prime \prime }}{|T_\eps^{^{\prime }}|T_\eps^{^{\prime }}}    }]\frac{1}{T_\eps^{\prime }}f^{\prime })
\end{split}
\end{equation}
and%
\begin{equation}\label{eq:s3}
\begin{split}
L_{T_\eps}(\frac{1}{T_\eps^{\prime }}(\frac{T_\eps^{\prime \prime }}{|T_\eps^{^{\prime }}|T_\eps^{^{\prime }}}%
f)^{\prime }) &=L_{T_\eps}(\frac{1}{T_\eps^{\prime }}[(\frac{T_\eps^{\prime \prime }}{%
{|T_\eps^{^{\prime }}| T_\eps^{^{\prime }}    }})^{\prime }f+(\frac{T_\eps^{\prime \prime }}{|T_\eps^{^{\prime }}| T_\eps^{^{\prime }}})f^{\prime }]) \\
&=L_{T_\eps}(\frac{1}{T_\eps^{\prime }}(\frac{T_\eps^{\prime \prime \prime }{|T_\eps^{^{\prime
}}| T_\eps^{^{\prime
}} }-{ 2(|T_\eps^{\prime }|T_\eps^{\prime \prime })T_\eps^{\prime \prime }}}{(T_\eps^{^{\prime
}})^{4}})f+\frac{1}{T_\eps^{\prime }}(\frac{T_\eps^{\prime \prime }}{|T_\eps^{^{\prime
}}|  T_\eps^{^{\prime }}   })f^{\prime })).
\end{split}
\end{equation}

{ Since $\beta \in (-1,-\frac{3}{4})$, by assumptions $(A4)$ and then $(A6)$, $(A8)$ we have, as $y\to c$}

\begin{itemize}
\item $(\frac{1}{T_\eps^{\prime }})^{2}\in L^{\infty };$

\item $g_{1,a}:=\frac{T_\eps^{\prime \prime }}{|T_\eps^{^{\prime }}|T_\eps^{^{\prime }}}\frac{1}{%
T_\eps^{\prime }}\sim |y-c|^{-2\beta -1}$; i.e, $g_{1,a}\in L^{\infty };$

{\item $g_{1,b}:=\frac{T_\eps^{\prime \prime }}{(T_\eps^{^{\prime }})^2}\frac{1}{T_\eps^{\prime }}\sim |y-c|^{-2\beta -1}$; i.e, $g_{1,b}\in L^{\infty };$}

\item $g_{2}:=\frac{1}{T_\eps^{\prime }}(\frac{T_\eps^{\prime \prime \prime
}{|T_\eps^{^{\prime }}|T_\eps^{^{\prime }}}}{(T_\eps^{^{\prime }})^{4}})\sim |y-c|^{-2\beta
-2}$, ; i.e.,  $g_{2}\in L^{2};$

\item $g_{3}:=\frac{1}{T_\eps^{\prime }}(\frac{ 2(|T_\eps^{\prime }|T_\eps^{\prime \prime
})T_\eps^{\prime \prime }}{(T_\eps^{^{\prime }})^{4}})\sim |y-c|^{-2-2\beta
} $; i.e.,  $g_{3}\in L^{2};$

\item $g_{4}:=\frac{T_\eps^{\prime \prime }}{|T_\eps^{^{\prime }}|T_\eps^{^{\prime }}}\frac{T_\eps^{\prime
\prime }}{ |T_\eps^{^{\prime }}|T_\eps^{^{\prime }}}\sim |y-c|^{-2-2\beta }$; i.e.,  
$g_{4}\in L^{2}.$
\end{itemize}
{ Furthermore, assumption $(A8)$ implies that $||g_{1,a}||_\infty$, $||g_{1,b}||_\infty$, $||g_{2}||_2$, $||g_{3}||_2$, $||g_{4}||_2$ are uniformly bounded, as $\varepsilon \in [0,\delta)$.
}

Consequently, for $f\in W^{2,1}$ we get $L_{T_\eps}f\in W^{2,1}$. More precisely, using the notation above in  \eqref{eq:s1}, considering \eqref{eq:s2} and \eqref{eq:s3}, we have
\begin{eqnarray*}
(L_{T_\eps}f)^{\prime \prime } &{=} &L_{T_\eps}((\frac{1}{T_\eps^{\prime }})^{2}f^{\prime
\prime }))-L_{T_\eps}({ [g_{1,a}+g_{1,b}]}f^{\prime
})-L_{T_\eps}((g_{2}-g_{3})f)-L_{T_\eps}(g_{ 1,a}f^{\prime })+L_{T_\eps}(g_{4}f) \\
&{=} &L_{T_\eps}((\frac{1}{T_\eps^{\prime }})^{2}f^{\prime \prime
}))-L_{T_\eps}({[2g_{1,a}+g_{1,b}]}f^{\prime })-L_{T_\eps}((g_{2}-g_{3}{-}g_{4})f),
\end{eqnarray*}
and then
\begin{equation}\label{eq:s4}
\begin{split}
\|(L_{T_\eps}f)^{\prime \prime }\|_{1} &\leq \|L_{T_\eps}((\frac{1}{T_\eps^{\prime }}
)^{2}f^{\prime \prime }))\|_{1}{+}\|L_{T_\eps}({ [2g_{1,a}+g_{1,b}]}f^{\prime
})\|_{1}{+}\|L_{T_\eps}((g_{2}-g_{3}{-}g_{4})f)\|_{1} \\
&\leq \lambda^{2}\|f^{\prime \prime
}\|_{1}{+}\|{ [2g_{1,a}+g_{1,b}]}\|_{\infty }\|f^{\prime
}\|_{1}{+} \|g_{2}-g_{3}{-}g_{4}\|_{2}\|f\|_{2}.
\end{split}
\end{equation}
{ Since the $W^{1,1}$ norm bounds the $L^2$ norm, on the compact space $[0,1]$ and as we observed, the norms of the various $g_i$ are uniformly bounded as $\varepsilon$ varies}, thus $\exists$ $M\geq0 $ such that {for each $\varepsilon \in[0,\delta)$}
\begin{equation}\label{eq:LYW21}
\|L_{T_\eps}f\|_{W^{2,1}}\leq\lambda^{2}\|f\|_{W^{2,1}}+ M\|f\|_{W^{1,1}}.
\end{equation}
\end{proof}

{
 We now show that   the transfer operators are continuous in the $L^2$ norm. 

\begin{lemma}\label{cont}
   For each $\varepsilon \in [0,\delta)$ we have
 \begin{equation}\label{eq:L2b}
 \|L_{T_\eps}\|_{L^2\to L^2}\leq2.
 \end{equation}
  \end{lemma}
  }
  \begin{proof}
  By  \eqref{rt} we have 
$$\|(L_{T_{\eps}}f)\|_{L^{2}}\leq \|\psi _{1,{\varepsilon }}\|_{L^{2}}+\|\psi
_{2,{\varepsilon }}\|_{L^{2}},$$ 
where 
\begin{equation}
\psi _{1,{\varepsilon }}(x):=(\frac{1}{|T_\eps'|}\cdot f)\circ T_{0,\eps}^{-1}(x)\cdot 1_{ T_{0,\eps}[0,c)}(x),
\label{eq:transeps}
\end{equation}
\begin{equation*}
\psi _{2,{\varepsilon }}(x):=(\frac{1}{|T_\eps'|}\cdot f)\circ T_{1,\eps}^{-1}(x)\cdot 1_{ T_{1,\eps}(c,1]}(x).
\end{equation*}%
Further,
\begin{eqnarray*}
\int_{\lbrack 0,1]}(\psi _{1,{\varepsilon }})^{2}dm
&=&\int_{[0,1]}[(\frac{1}{|T_{\eps}^{\prime }|}\cdot f)\circ
T_{0,\eps}^{-1}(x)\cdot 1_{T_{0,\eps}[0,c)}(x)]^{2}dx \\
&=&\int_{[0,a_{\varepsilon }]}\frac{1}{|T_{\eps}^{\prime }(T_{0,\eps%
}^{-1}(x))|}\frac{1}{|T_{\eps}^{\prime }(T_{0,\eps}^{-1}(x))|}\cdot \lbrack
f(T_{0,\eps}^{-1}(x))]^{2}dx \\
&\leq &\sup_x [\frac{1}{|T^{\prime }(x)|}]\int_{[0,a_{\varepsilon }]}\frac{1}{|T_{%
\eps}^{\prime }(T_{0,\eps}^{-1}(x))|}\cdot \lbrack f(T_{0,\eps%
}^{-1}(x))]^{2}dx \\
&=&\sup_x [\frac{1}{|T^{\prime }(x)|}]\|L_{T_{\eps}}(f^{2}\cdot
1_{[0,c)})\|_{L^{1}}\leq \sup_x [\frac{1}{|T^{\prime }(x)|}]\|f^{2}\cdot
1_{[0,c)}\|_{L^{1}} \\
&\leq &\sup_x [\frac{1}{|T^{\prime }(x)|}][\|f\|_{L^{2}}]^{2}.
\end{eqnarray*}
A similar estimate holds for $\|\psi
_{2,{\varepsilon }}\|^2_{L^{2}}$.
\end{proof}

\begin{proof}[Proof of Proposition \ref{lem:spec}]
{Using Lemma \ref{cont}, the Lasota Yorke inequalities proved in Lemma \ref{lem:LYW11} and the fact that $W^{1,1}  $ }is  compactly embedded in $ L^{2}$ by the Rellich-Kondrakov theorem, applying Lemma \ref{BGK},  we get that the essential  spectral radius $\rho_{\text{ess}}\le\lambda<1$.
By this, any element of the spectrum with modulus strictly bigger than $\lambda $ is an isolated eigenvalue. 
We now prove that spectral radius of $L_{T_\eps}$, as an operator on $W^{1,1}$, is $1$. 

Since $T\circ\bf{1}={\bf 1}$; $1$ is in the spectrum of $L_{T_\eps}$ and in fact it is an eigenvalue (since $\rho_{\text{ess}}\le\lambda<1)$. Moreover, $L_{T_\eps}$ cannot have eigenvalues $\rho$ with $|\rho|>1$ since $ L_{T_\eps}$  is positive and it preserves integrals (see \cite{BG} properties of transfer operators). Therefore $1$ is the spectral radius on $L_{T_\eps}$ as an operator on $W^{1,1}$.

We now show that there are no other eigenvalues on the unit circle and that the eigenvalue $1$ is simple. In the case of $T_0$ this is true since in (A5) we assume $T_0$ to be topologically mixing \cite{V}. Thus, $L_{T_0}$ has a spectral gap on $W^{1,1}$. Since for any $\eps\in[0,\delta)$, we proved in Lemma \ref{lem:LYW11} a uniform (in $\eps$) Lasota Yorke inequality for $L_{T_\eps}$, by  (A9) \eqref{near}, the Keller-Liverani \cite{KL} spectral perturbation result implies that for sufficiently small {$0<\eps \leq \delta_2$, } the spectral projections corresponding to isolated eigenvalues of $L_{T_\eps}$ are in one to one correspondence with those of $L_{T_0}$. Thus, $1$ is also a simple eigenvalue for $L_{T_\eps}$, for {$0<\eps \leq \delta_2$}, and $L_{T_\eps}$ as an operator on $W^{1,1}$ has no other eigenvalues on the unit circle. This in particular implies that $T_\eps$ admits a unique invariant probability density $h_\eps\in W^{1,1}$.



By Lemma \ref{lem:LYW11},  $L_{T_\eps}$ is bounded when acting on $W^{1,1}$. Moreover,  $W^{2,1}$ is compactly embedded in $W^{1,1}$ { and thus} Lemma \ref{BGK} also implies the essential spectral radius of $L_{T_\eps}$  when acting on ${W^{2,1}}$ is smaller than $\lambda$. By the same reasoning as before we get that the spectral radius is $1$. Since we have already proved that $h_\eps$ is the unique invariant density in $W^{1,1}$, it follows that $h_\eps$ is the unique invariant probability density in ${W^{2,1}}$ and by the uniform Lasota Yorke inequality proved in Lemma  \ref{l:LYW21} we have a uniform bound on its $W^{2,1}$ norm.

To prove \eqref{resnorm} we apply Lemma \ref{KL1} to $L_{T_\eps}$ when acting on $W^{1,1}_0$. Considering $\| \  \|_{L^2}$ as the weak norm.  First notice that $1$ is not in the spectrum of $L_{T_0}:W^{1,1}_0\to W^{1,1}_0$ and by Lemma  \ref{BGK} the essential spectral radius of $L_{T_\eps}$ is bounded by $\lambda$. By Lemma \ref{lem:LYW11} the operators $L_{T_\eps}$ indeed satisfy a uniform Lasota-Yorke Inequality.  iterating this inequality and using \eqref{eq:L2b} we get 
\begin{equation}
\|L_{T_\eps}^n f\|_{W^{1,1}}\leq \lambda^n \|f\|_{W^{1,1}}+M  2^n\|f\|_{w}
\end{equation} verifying assumption \eqref{ULY}. The assumption \eqref{eq:KLmixed} is verified in (A9) \eqref{near}, in particular, 
$$\|L_{T_\eps}-L_{T_0}\|_{W^{1,1}\to L^2}=\sup_{\|f\|_{W^{1,1}}\le1}\|(L_{T_\eps}-L_{T_0})f\|_{L^2}\to0$$
by (A9) \eqref{near}. The application of Lemma  \ref{KL1} gives then  directly \eqref{resnorm}.
\end{proof}

\subsection{Linear response derivation}\label{subsec:resp}

\begin{lemma}
${\varepsilon}\mapsto h_{\varepsilon}$ is differentiable, at $\eps=0$, in $L^1$. In particular, 
\begin{equation*}
h_{\varepsilon} = h_0 + {\varepsilon}(I-{L_{T_0}})^{-1}(q) + o({\varepsilon}%
),
\end{equation*}
Where 
\begin{equation}\label{eq:q}
q(x)=\left\{ 
\begin{array}{lr}
L_{T_0} [A_0 h_0^{\prime }+ B_0 h_0](x) &\text { for }x\in [0,a_0)\\
0&\text{ for } x\in  [a_0,1]
\end{array}
\right. .  
\end{equation}
with
\begin{equation*}
A_{\varepsilon}=-\left(\frac{\partial_{{\varepsilon}}T_{{\varepsilon}}}{
T^{\prime }_{{\varepsilon}}}\right),\hskip 0.5cm B_{\varepsilon}= \left(%
\frac{\partial_{{\varepsilon}} T_{{\varepsilon}}\cdot T_{{\varepsilon}%
}^{\prime \prime }}{T_{\varepsilon}^{\prime 2}}-\frac{\partial_{{\varepsilon}%
} T_{{\varepsilon}}^{\prime }}{ T_{\varepsilon}^{\prime }}\right)
\end{equation*}
and the $o$ is in the $L^1$-topology.
\end{lemma}
\begin{proof}
We introduce the following notation: 
\begin{equation*}
H_{\varepsilon }:=L_{T_\varepsilon }-L_{T_0};\quad
G_{\varepsilon }:=(I-L_{T_\varepsilon })^{-1}.
\end{equation*}%
Since $L_{T_\varepsilon }$ has a spectral gap on $W^{1,1}$ it
eventually contracts exponentially on the subset of zero average functions $%
W^{1,1}_{0}$ and the following relation is well defined: 
\begin{equation}
h_{\varepsilon }=G_{\varepsilon }H_{\varepsilon }h_0+h_0.  \label{eq:GH}
\end{equation}%
 Recall that $h_0\in W^{2,1}$. Therefore, the second part of assumption $(A9)$ we have
\begin{equation}\label{eq:der1}
\Vert L_{T_\eps}h_0-L_{T_0}h_0-{\varepsilon }(\partial
_{\eps }L_{T_\eps}h_0|_{\eps =0 })\Vert
_{W^{1,1}}=o(\eps).
\end{equation}
Thus, by \eqref{eq:der1}, we have
\begin{equation}
H_{\varepsilon }h={\varepsilon }q+o({\varepsilon }), \label{eq:step1}
\end{equation}
for some $q\in W_0^{1,1}$, with the error $o(\eps)$ understood in $W^{1,1}$. To obtain a formula for $q$, for $x\in [0,a_\eps)$, set $g_{j,\eps}:=T_{j,\eps}^{-1}(x)$ and consider
\[
\begin{split}
\partial_\eps (h\circ g_{j,\eps}g_{j,\eps}')
&=
\partial_\eps(h\circ g_{j,\eps})g_{j,\eps}' + h\circ g_{j,\eps} \partial_\eps g_{j,\eps}'\\
&=
h'\circ g_{j,\eps} \partial_\eps g_{j,\eps}g_{j,\eps}' + h\circ g_{j,\eps} \partial_\eps g_{j,\eps}'.
\end{split}
\]
To prove the statement for $x\in[0,a_\eps)$, we start from the relation $T_\eps \circ g_{j,\eps}(x)=x$ and differentiate it  with respect to $\eps$ and get $T_\eps'\circ g_{j,\eps} \partial_\eps g_{j,\eps} + \partial_\eps T_\eps \circ g_{j,\eps}=0$. This gives $\partial_\eps g_{j,\eps} = A_\eps\circ g_{j,\eps}$.
This also implies that $\partial_\eps g_{j,\eps}' = A_\eps' \circ g_{j,\eps} g_{j,\eps}' = B_\eps\circ g_{j,\eps} g_{j,\eps}'$. This provides the formula for $q$ in \eqref{eq:q}.

To continue, recall that $L_{T_\eps}$ admits a uniform, in $\eps$, spectral gap on $W^{1,1}$. Therefore, $G_{\varepsilon }$ is uniformly bounded  in $\mathcal L (W_{0}^{1,1},W^{1,1})$ and we  have
\begin{equation}
G_{\varepsilon }H_{\varepsilon }h={\varepsilon }G_{\varepsilon }q+o({
\varepsilon }),  \label{eq:step2}
\end{equation}
where the above error is understood in $W^{1,1}$. 
Moreover, by the stability result of \cite{KL}, we get
 \begin{equation}
\lim_{{\varepsilon }\rightarrow 0}\Vert G_{\varepsilon }(q)-G_{0}(q)\Vert
_{L^1}=0.  \label{eq:step3}
\end{equation}%
Using \eqref{eq:step1}, \eqref{eq:step2} and \eqref{eq:step3} together with 
\eqref{eq:GH} we obtain in $L^1$ 
\begin{equation*}
h_{\varepsilon }=h_0+{\varepsilon }G_{0}(q)+o({\varepsilon }),
\end{equation*}%
which proves differentiability of $h_{\varepsilon }$ and completes the proof
of the theorem.
\end{proof}

\section{A family of maps satisfying  assumptions $(A1),...,(A9)$}\label{example}

\begin{figure}[h]
    \centering
\includegraphics[scale=0.5]{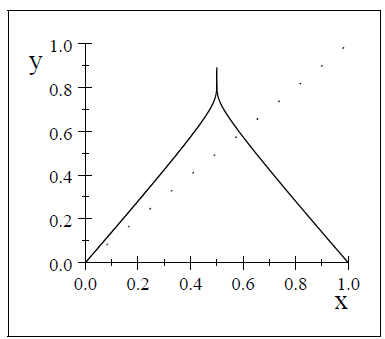}
    \caption{An example of $T_\varepsilon$ with $a_\varepsilon=0.89$ and $c=\frac{1}{2}$.}
    \label{fig:mesh1}
\end{figure}

We provide an explicit example of a family of maps satisfying assumptions $%
(A1),...,(A9)$. Let $\varepsilon \in \lbrack 0,\frac{1}{10})$ and 
\begin{equation}\label{eq:familyex}
T_{\varepsilon }(x)=\left\{ 
\begin{array}{c}
(1-\varepsilon )(\frac{3}{4}(2x)+\frac{1}{4}(1-\sqrt[8]{1-2x}))~\text{ for }~x\in
\lbrack 0,\frac{1}{2}) \\ 
(1-\varepsilon )(\frac{3}{4}(2-2x)+\frac{1}{4}(1-\sqrt[8]{2x-1}))~\text{ for }~x\in (%
\frac{1}{2},1].
\end{array}%
\right. 
\end{equation}%
The graph of one member of this family is shown in Figure \ref{fig:mesh1}.
In this example $c=\frac{1}{2}$ and it is immediate to see that $%
T_{\varepsilon }$ satisfies $(A1),(A2),(A3),$ $(A5).$ Furthermore 

\begin{equation*}
T_{\varepsilon }^{\prime }(x)=\left\{ 
\begin{array}{c}
\left( 1-\varepsilon \right) \left( \frac{1}{16(1-2x)}\left( \sqrt[8]{1-2x}%
+24(1-2x)\right) \right) ~\text{for}~x\in \lbrack 0,\frac{1}{2}) \\ 
\left( 1-\varepsilon \right) \left( \frac{-1}{16(2x-1)}\left( \sqrt[8]{2x-1}%
+24(2x-1)\right) \right) ~\text{for}~x\in (\frac{1}{2},1]%
\end{array}%
\right.
\end{equation*}%
and $|T_{\varepsilon }^{\prime }(x)|\geq \frac{45}{32}$ for $x\in \lbrack 0,%
\frac{1}{2})\cup (\frac{1}{2},1]$ and $\varepsilon \in \lbrack 0,\frac{1}{10})$%
. Thus, $T_{\varepsilon }$ satisfies $(A4)$. Also $T_{\varepsilon }^{\prime }$
satisfies $(A6)$ with $\beta =\frac{-7}{8}.$ Computing the second derivative we get
\begin{equation*}
T_{\varepsilon }^{\prime \prime }(x)=\left\{ 
\begin{array}{c}
\frac{7\left( 1-\varepsilon \right) }{64\left( 2x-1\right) ^{2}}\sqrt[8]{1-2x%
}~\text{fo}r~x\in \lbrack 0,\frac{1}{2}) \\ 
\frac{7\left( 1-\varepsilon \right) }{64\left( 2x-1\right) ^{2}}\sqrt[8]{2x-1%
}\allowbreak ~\text{for}~x\in (\frac{1}{2},1]%
\end{array}%
\right.
\end{equation*}%
and thus the first part of $(A7)$ is verified. Similarly one can proceed
with the third derivative and verify the second part of $(A7)$ and $(A8).$ We now verify $(A9)$. Let $T_{0,\varepsilon }$ and $T_{1,\varepsilon }$ be the
branches of $T_{\varepsilon }$, as in $(A1)$. We have $T_{i,\varepsilon
}=D_{\varepsilon }\circ T_{i,0}$, where $D_{\varepsilon }(x)=\left( 1-\varepsilon
\right) (x)$. Let $L_{D_{\varepsilon }},~L_{T_{0}}$ denote the transfer
operators associated with $D_{\varepsilon }$ and $T_{0}$ respectively. Note that
have 
$$L_{T_{\varepsilon }}=L_{D_{\varepsilon }}L_{T_{0}}$$ 
and for any $x\in \lbrack 0,1]$, 
$$L_{D_{\varepsilon
}}g(x)=\left\{ 
\begin{array}{c}
\left( 1-\varepsilon \right) ^{-1}g(\left( 1-\varepsilon \right) ^{-1}x)~\text{ for }~x\in (1-\varepsilon ]\\
 0 ~\text{ for }~x\in(1-\varepsilon,1].
\end{array}%
\right. $$
First we verify that
\begin{equation*}
\sup_{\Vert f\Vert _{W^{1,1}}\leq 1}\Vert (L_{T_{0}}-L_{T_{\varepsilon
}})f\Vert _{L^{2}}\rightarrow 0
\end{equation*}%
as $\varepsilon \rightarrow 0.$ Since we have already verified that $%
T_{\eps}$ satisfies $(A1),...,(A8),$ by Lemma 4 we obtain that $%
g:=L_{T_{0}}f\in W^{1,1}$ when $f\in W^{1,1}$. Furthermore there is an $M>0$ such
that $\Vert f\Vert _{W^{1,1}}\leq 1$ implies $\Vert g\Vert _{W^{1,1}}\leq M$. Then, it is sufficient to prove that
\begin{equation} \label{l2}
\sup_{\Vert g\Vert _{W^{1,1}}\leq M}\Vert g-L_{D_{\varepsilon
}}g||_{L^{2}}\rightarrow 0 
\end{equation}
as $\varepsilon \rightarrow 0.$
 For brevity we will also denote $\ L_{D_{\varepsilon }}g$ as
\begin{equation}\label{eq:notation}
L_{D_{\varepsilon }}g(x)=\left( 1-\varepsilon \right) ^{-1}1_{[0,1-\varepsilon
]}g(\left( 1-\varepsilon \right) ^{-1}x).
\end{equation}
We have
\begin{eqnarray*}
\Vert g-L_{D_{\varepsilon }}g||_{L^{2}}^{2} &=&\int_{0}^{1}(g-L_{D_{\varepsilon
}}g)(g-L_{D_{\varepsilon }}g)dm \\
&\leq &4M \int_{0}^{1}|g-L_{D_{\varepsilon }}g|dm \\
&\leq &4M \int_{0}^{1}|g(x)-\left( 1-\varepsilon \right)
^{-1}1_{[0,1-\varepsilon ]}(x)g(\left( 1-\varepsilon \right) ^{-1}x)|dx \\
&\leq &4M\int_{0}^{1-\varepsilon }|g(x)-\left( 1-\varepsilon \right)
^{-1}g(\left( 1-\varepsilon \right) ^{-1}x)|dx +8M^2\varepsilon 
\end{eqnarray*}

since $||g||_{\infty }\leq 2||g||_{W^{1,1}}$. Moreover,
\begin{equation*}
\int_{0}^{1-\varepsilon }|g(x)-\left( 1-\varepsilon \right) ^{-1}g(\left(
1-\varepsilon \right) ^{-1}x)|dx\leq 2M[1-\left( 1-\varepsilon \right)
^{-1}]+\int_{0}^{1-\varepsilon }|g(x)-g(\left( 1-\varepsilon \right) ^{-1}x)|dx
\end{equation*}%
and%
\begin{eqnarray}
\int_{0}^{1-\varepsilon }|g(x)-g(\left( 1-\varepsilon \right) ^{-1}x)|dx
&=&\int_{0}^{1-\varepsilon }|\int_{x}^{\left( 1-\varepsilon \right)
^{-1}x}g^{\prime }(t)dt|dx  \label{sopra} \\
&\leq &\int_{0}^{1-\varepsilon }\int_{x}^{\left( 1-\varepsilon \right)
^{-1}x}|g^{\prime }(t)|dtdx.  \notag
\end{eqnarray}

Now changing the order of integration we get 
\begin{eqnarray}
\int_{0}^{1-\varepsilon }\int_{x}^{\left( 1-\varepsilon \right) ^{-1}x}|g^{\prime
}(t)|dtdx &\leq &\int_{0}^{1}\int_{t-\varepsilon t}^{t}|g^{\prime }(t)|dxdt
\label{scambio} \\
&\leq &\int_{0}^{1}\varepsilon t|g^{\prime }(t)|dt
\end{eqnarray}%
\newline
The estimate in \eqref{scambio} is explained in Figure \ref{fig:scambio} where the domains of integration are shown. We note that the left hand side integral in \eqref{scambio} is considered over
the domain $OAB$ in Figure \ref{fig:scambio} while the right hand side integral is considered over the
larger domain $OAC$ in the same figure. Therefore,
\begin{figure}[h]
\label{fig:scambio}
\centering
\includegraphics[scale=0.5]{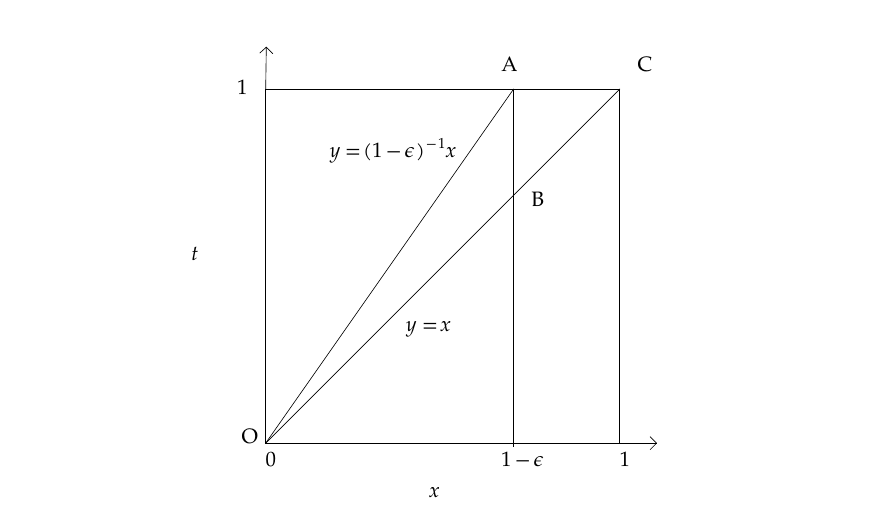}
\caption{The integration domains in \eqref{scambio}. }
\end{figure}
\begin{eqnarray*}
\int_{0}^{1-\varepsilon }|g(x)-g(\left( 1-\varepsilon \right) ^{-1}x)|dx &\leq
&\int_{0}^{1}\varepsilon t|g^{\prime }(t)|dt \\
&\leq &\varepsilon ||g^{\prime }||_{L^{1}} \\
&\leq &\varepsilon ||g||_{W^{1,1}}\rightarrow 0
\end{eqnarray*} 
and thus $(\ref{l2})$ is verified.

Now for $f\in W^{2,1}$ we prove that the limit
\begin{equation*}
\lim_{\varepsilon \rightarrow 0}\frac{1}{\varepsilon }(L_{T_{0}}-L_{T_{\varepsilon
}})f
\end{equation*}
exists in $W^{1,1}$. By the same reasoning as above,
using Lemma \ref{l:LYW21}, and using the notation \eqref{eq:notation}, we get this reduced to proving the
existence of the limit
\begin{equation*}
\lim_{\varepsilon \rightarrow 0}\frac{L_{D_{\varepsilon }}g-g}{\varepsilon }
=\lim_{\varepsilon \rightarrow 0}\frac{\left( 1-\varepsilon \right)
^{-1}1_{[0,1-\varepsilon ]}(x)g(\left( 1-\varepsilon \right) ^{-1}x)-g(x)}{
\varepsilon }
\end{equation*}
in $W^{1,1}$, where $g=L_{T_{0}}f$ as above\ and $g\in W^{2,1}$. \ Now we
have for each $x\in \lbrack 0,1]$
\begin{eqnarray*}
\frac{\left( 1-\varepsilon \right) ^{-1}1_{[0,1-\varepsilon ]}(x)g(\left(
1-\varepsilon \right) ^{-1}x)-g(x)}{\varepsilon } &=&\frac{1_{[0,1-\varepsilon
]}g(\left( 1-\varepsilon \right) ^{-1}x)-g(x)}{\varepsilon } \\
&&\hskip-3cm+\frac{\left( 1-\varepsilon \right) ^{-1}1_{[0,1-\varepsilon ]}(x)g(\left(
1-\varepsilon \right) ^{-1}x)-1_{[0,1-\varepsilon ]}(x)g(\left( 1-\varepsilon \right)
^{-1}x)}{\varepsilon } \\
&&\hskip-3cm:=I+II.
\end{eqnarray*}

\begin{eqnarray*}
II &=&\frac{\left( 1-\varepsilon \right) ^{-1}1_{[0,1-\varepsilon ]}(x)g(\left(
1-\varepsilon \right) ^{-1}x)-1_{[0,1-\varepsilon ]}(x)g(\left( 1-\varepsilon \right)
^{-1}x)}{\varepsilon } \\
&=&\frac{(1+~[\frac{\varepsilon }{1-\varepsilon }])1_{[0,1-\varepsilon ]}(x)g(\left(
1-\varepsilon \right) ^{-1}x)-1_{[0,1-\varepsilon ]}(x)g(\left( 1-\varepsilon \right)
^{-1}x)}{\varepsilon } \\
&=&\frac{1_{[0,1-\varepsilon ]}(x)g(\left( 1-\varepsilon \right) ^{-1}x)}{%
1-\varepsilon } \\
&\rightarrow &g(x).
\end{eqnarray*}%
Since $g\in W^{2,1}\subseteq W^{1,1}$ the convergence in $L^{1}$ of this
limit can be proved as in $(\ref{sopra}).$ Now considering the derivative,
since $L_{D_{\varepsilon }}g(x)\in W^{2,1}$ we have%
\begin{eqnarray*}
\lbrack 1_{[0,1-\varepsilon ]}(x)g(\left( 1-\varepsilon \right) ^{-1}x)]^{\prime }
&=&\left\{ 
\begin{array}{c}
\left( 1-\varepsilon \right) ^{-1}g^{\prime }(\left( 1-\varepsilon \right)
^{-1}x)\text{ for }x\in \lbrack 0,1-\varepsilon ] \\ 
0\text{ for }x\in (1-\varepsilon, 1 ]%
\end{array}%
\right. \\
&=&\left( 1-\varepsilon \right) ^{-1}1_{[0,1-\varepsilon ]}(x)g^{\prime }(\left(
1-\varepsilon \right) ^{-1}x)
\end{eqnarray*}
then%
\begin{equation*}
\lbrack \frac{1_{[0,1-\varepsilon ]}(x)g(\left( 1-\varepsilon \right) ^{-1}x)}{%
1-\varepsilon }-g(x)]^{\prime }=\frac{1_{[0,1-\varepsilon ]}(x)g^{\prime }(\left(
1-\varepsilon \right) ^{-1}x)}{(1-\varepsilon )^{2}}-g^{\prime }(x)
\end{equation*}%
since $g^{\prime }\in W^{1,1}$ we can repeat the same reasoning as in \eqref{scambio} and then get
the convergence in $L^{1}$ of the derivative too, leading to the $W^{1,1}$
convergence of the limit considered in $II$.

Now we consider the limit related to $I:$%
\begin{eqnarray*}
I &=&\frac{1_{[0,1-\varepsilon ]}(x)g(\left( 1-\varepsilon \right) ^{-1}x)-g(x)}{%
\varepsilon } \\
&=&\frac{1_{[0,1-\varepsilon ]}(x)g(x+\frac{\varepsilon }{1-\varepsilon }x)-g(x)}{%
\varepsilon } \\
&\rightarrow &xg^{\prime }(x)
\end{eqnarray*}%
for almost every $x$. To prove convergence in $L^{1}$, let us consider $%
h(x)\in C^{2}[0,1]$ with $h(0)=g(0)$ \ and $||g^{\prime }-h^{\prime
}||_{L^{1}}\leq \varepsilon _{2}$, where $\varepsilon _{2}>0$ is arbitrary. Observe that
\begin{eqnarray*}
Err(\varepsilon )&:=&\int_{0}^{1-\varepsilon }|\frac{1_{[0,1-\varepsilon ]}(x)g(x+%
\frac{\varepsilon }{1-\varepsilon }x)-g(x)}{\varepsilon }-xg^{\prime }(x)|dx \\
&=&\int_{0}^{1-\varepsilon }|\frac{\int_{x}^{x+\frac{\varepsilon }{1-\varepsilon }%
x}g^{\prime }(t)dt}{\varepsilon }-xg^{\prime }(x)|dx \\
&=&\int_{0}^{1-\varepsilon }|\frac{\int_{x}^{x+\frac{\varepsilon }{1-\varepsilon }%
x}g^{\prime }(t)dt-\int_{x}^{x+\frac{\varepsilon }{1-\varepsilon }x}h^{\prime
}(t)dt+\int_{x}^{x+\frac{\varepsilon }{1-\varepsilon }x}h^{\prime }(t)dt}{\varepsilon 
}-xg^{\prime }(x)|dx \\
&\leq &\int_{0}^{1-\varepsilon }\frac{|\int_{x}^{x+\frac{\varepsilon }{1-\varepsilon }%
x}g^{\prime }(t)dt-\int_{x}^{x+\frac{\varepsilon }{1-\varepsilon }x}h^{\prime
}(t)dt|}{\varepsilon }+|\frac{\int_{x}^{x+\frac{\varepsilon }{1-\varepsilon }%
x}h^{\prime }(t)dt}{\varepsilon }-xg^{\prime }(x)|dx.
\end{eqnarray*}%
We have 
\begin{equation*}
\int_{0}^{1-\varepsilon }\frac{|\int_{x}^{x+\frac{\varepsilon }{1-\varepsilon }%
x}g^{\prime }(t)dt-\int_{x}^{x+\frac{\varepsilon }{1-\varepsilon }x}h^{\prime
}(t)dt|}{\varepsilon }dx\leq \frac{\int_{0}^{1-\varepsilon }\int_{x}^{x+\frac{%
\varepsilon }{1-\varepsilon }x}|g^{\prime }(t)-h^{\prime }(t)|dtdx}{\varepsilon }.
\end{equation*}%
Now changing the order of integration as in $(\ref{scambio})$ we get 
\begin{equation*}
\frac{\int_{0}^{1-\varepsilon }\int_{x}^{x+\frac{\varepsilon }{1-\varepsilon }%
x}|g^{\prime }(t)-h^{\prime }(t)|dtdx}{\varepsilon }\leq \frac{%
\int_{0}^{1}\int_{t-\varepsilon t}^{t}|g^{\prime }(t)-h^{\prime }(t)|dxdt}{%
\varepsilon }.
\end{equation*}%
By this%
\begin{eqnarray*}
\int_{0}^{1-\varepsilon }\frac{|\int_{x}^{x+\frac{\varepsilon }{1-\varepsilon }%
x}g^{\prime }(t)dt-\int_{x}^{x+\frac{\varepsilon }{1-\varepsilon }x}h^{\prime
}(t)dt|}{\varepsilon }dx &\leq &\frac{\int_{0}^{1}\int_{t-\varepsilon
t}^{t}|g^{\prime }(t)-h^{\prime }(t)|dxdt}{\varepsilon } \\
&\leq &\frac{\int_{0}^{1}\varepsilon t|g^{\prime }(t)-h^{\prime }(t)|dt}{%
\varepsilon } \\
&\leq &\varepsilon _{2}
\end{eqnarray*}%
where the last inequality holds since \thinspace $0\leq t\leq 1$. On the other hand%
\begin{eqnarray*}
&&\int_{0}^{1-\varepsilon }|\frac{\int_{x}^{x+\frac{\varepsilon }{1-\varepsilon }%
x}h^{\prime }(t)dt}{\varepsilon }-xg^{\prime }(x)|dx \\
&&=\int_{0}^{1-\varepsilon }|%
\frac{\int_{x}^{x+\frac{\varepsilon }{1-\varepsilon }x}h^{\prime }(t)dt}{\varepsilon }%
-xh^{\prime }(x)-xg^{\prime }(x)+xh^{\prime }(x)|dx \\
&&\le\int_{0}^{1-\varepsilon }|\frac{\int_{x}^{x+\frac{\varepsilon }{1-\varepsilon }%
x}h^{\prime }(t)dt}{\varepsilon }-xh^{\prime }(x)|dx \\
&&+\int_{0}^{1-\varepsilon }|xg^{\prime }(x)-xh^{\prime }(x)|dx.
\end{eqnarray*}
Now we have $\int_{0}^{1-\varepsilon }|xg^{\prime }(x)-xh^{\prime }(x)|dx\leq
\varepsilon _{2}$ as before, and by uniform continuity of $h^{\prime }$ there
is a function $l(\varepsilon )$ such that $l(\varepsilon )\rightarrow 0$ as $%
\varepsilon \rightarrow 0,$ such that $|h^{\prime }(t)-h^{\prime }(x)|\leq
l(\varepsilon )$ for $t\in \lbrack x,x+\frac{\varepsilon }{1-\varepsilon }x]$, $x\in
\lbrack 0,1-\varepsilon ]$, then
\begin{eqnarray*}
&&\int_{0}^{1-\varepsilon }|\frac{\int_{x}^{x+\frac{\varepsilon }{1-\varepsilon }%
x}h^{\prime }(t)dt}{\varepsilon }-xh^{\prime }(x)|dx\\
&&\leq\int_{0}^{1-\varepsilon }|\frac{\int_{x}^{x+\frac{\varepsilon }{1-\varepsilon }%
x}h^{\prime }(t)-h^{\prime }(x)+h^{\prime }(x)dt}{\varepsilon }-xh^{\prime
}(x)|dx \\
&&\le \int_{0}^{1-\varepsilon }\frac{\int_{x}^{x+\frac{\varepsilon }{1-\varepsilon }%
x}|h^{\prime }(t)-h^{\prime }(x)|dt}{\varepsilon }dx \\
&&+\int_{0}^{1-\varepsilon }|\frac{1}{1-\varepsilon }xh^{\prime }(x)-xh^{\prime
}(x)|dx \\
&&\le\int_{0}^{1-\varepsilon }\frac{\int_{x}^{x+\frac{\varepsilon }{1-\varepsilon }%
x}l(\varepsilon )dt}{\varepsilon }dx+o(1)=o(1)
\end{eqnarray*}%
and $Err(\varepsilon )\leq 2\varepsilon _{2}+o(1)$ \ with \ $\varepsilon _{2}$ being
arbitrary and \ $o(1)\rightarrow 0$ as $\varepsilon \rightarrow 0$, hence $%
Err(\varepsilon )\rightarrow 0$ as $\varepsilon \rightarrow 0$ and we established
the $L^{1}$ convergence of $I$ to the function $x\rightarrow xg^{\prime }(x)$%
. \ Notice that we only used $g\in W^{1,1}$ in this proof. But for $g\in
W^{2,1}$ then $g^{\prime }\in W^{1,1}$. This will be used in the next step,
as now we are going to show the convergence of the same limit for the derivative. Indeed, we have
\begin{eqnarray}
I^{\prime } &=&[\frac{1_{[0,1-\varepsilon ]}(x)g(\left( 1-\varepsilon \right)
^{-1}x)-g(x)}{\varepsilon }]^{\prime }  \label{2} \\
&=&\frac{\left( 1-\varepsilon \right) ^{-1}1_{[0,1-\varepsilon ]}(x)g^{\prime
}(\left( 1-\varepsilon \right) ^{-1}x)-g^{\prime }(x)}{\varepsilon }  \notag \\
&=&\frac{1_{[0,1-\varepsilon ]}(x)g^{\prime }(\left( 1-\varepsilon \right)
^{-1}x)-g^{\prime }(x)}{\varepsilon }  \notag \\
&&+\frac{\left( 1-\varepsilon \right) ^{-1}1_{[0,1-\varepsilon ]}(x)g^{\prime
}(\left( 1-\varepsilon \right) ^{-1}x)-1_{[0,1-\varepsilon ]}(x)g^{\prime }(\left(
1-\varepsilon \right) ^{-1}x)}{\varepsilon }  \notag \\
&\rightarrow &g^{\prime }(x)+xg^{\prime \prime }(x)=[xg^{\prime
}(x)]^{\prime }  \notag
\end{eqnarray}%
which can be treated as before noting the similarity between the third
line of $(\ref{2})$ and $I$ and the similarity between the fourth line of $(%
\ref{2})$ and $II$. This concludes the verification of assumption $(A9)$
for the family of maps in \eqref{eq:familyex}.
\section{Concluding remarks}\label{Sec5}
We presented a class of tent-like maps with singulartities where the presence of a cusp with a certain power law behavior induces a regularization at the level of the transfer operator associated to the map. This leads to a spectral gap of the transfer operator when acting on suitable Sobolev spaces. Unlike perturbations of tent maps in the bounded derivative case, we have shown that this induces linear response of the invariant density under deterministic perturbations of the system that changes the image of the critical point.

It would be interesting to explore how far regularization effect of singularities can play a role in obtaining linear response for more general systems with singularities, and perhaps with discontinuities. 
To implement a suitable generalization of the approach of this paper in more general settings, further technical work is needed (see for instance Remark \ref{rmk01} for a comment on the suitable choice of the weak space). Such generalizations are strongly motivated by possible applications to Lorenz-like systems and billiards, as remarked at the end of the introduction.

\end{document}